\documentclass[11pt,twoside]{amsart}
\usepackage{amsxtra}
\usepackage{color}
\usepackage{amsopn}
\usepackage{amsmath,amsthm,amssymb}
\usepackage{mathrsfs,mathtools,mathabx}
\usepackage{hyperref}
\usepackage{enumerate}
\usepackage{xcolor}
\usepackage{pifont}
\usepackage{cancel}

\newtheorem{theorem}{Theorem}[section]
\newtheorem{corollary}[theorem]{Corollary}
\newtheorem{proposition}[theorem]{Proposition}
\newtheorem{lemma}[theorem]{Lemma}
\newtheorem*{theorem*}{Theorem}

\theoremstyle{definition}

\makeatletter
\@namedef{subjclassname@2020}{%
  \textup{2020} Mathematics Subject Classification}
\makeatother

\def\sideremark#1{\ifvmode\leavevmode\fi\vadjust{
		\vbox to0pt{\hbox to 0pt{\hskip\hsize\hskip1em
				\vbox{\hsize3cm\tiny\raggedright\pretolerance10000
					\noindent #1\hfill}\hss}\vbox to8pt{\vfil}\vss}}}

\newcommand{\R}{{\mathbb R}}
\renewcommand{\H}{\mathrm{H}}

\newcommand{\beq}{\begin{equation}}
\newcommand{\eeq}{\end{equation}}
\renewcommand{\a}{\alpha}
\renewcommand{\b}{\beta}

\newcommand{\f}{\varphi}

\renewcommand{\o}{\omega}

\newcommand{\U}{{\mathrm U}}
\newcommand{\SU}{{\mathrm{SU}}}

\newcommand{\SO}{{\mathrm {SO}}}

\newcommand{\G}{{\mathrm G}}

\newcommand{\Sg}{{\mathrm S}}

\newcommand{\M}{{\mathrm M}}

\newcommand{\W}{\wedge}
\newcommand{\Diff}{\mathrm{Dif{}f}}

\DeclareMathOperator\Aut{Aut}

\DeclareMathOperator\Ad{Ad}
\DeclareMathOperator\ad{ad}

\DeclareMathOperator{\Span}{Span}

\newcommand{\frg}{\mathfrak{g}}

\newcommand{\frh}{\mathfrak{h}}

\newcommand{\ga}{\mathfrak{a}}
\newcommand{\gb}{\mathfrak{b}}

\renewcommand{\gg}{\mathfrak{g}}
\newcommand{\gh}{\mathfrak{h}}
\newcommand{\gk}{\mathfrak{k}}
\newcommand{\gl}{\mathfrak{l}}
\newcommand{\gm}{\mathfrak{m}}
\newcommand{\gn}{\mathfrak{n}}

\newcommand{\gp}{\mathfrak{p}}
\newcommand{\gq}{\mathfrak{q}}

\newcommand{\gs}{\mathfrak{s}}

\newcommand{\gu}{\mathfrak{u}}
\newcommand{\gv}{\mathfrak{v}}

\newcommand{\gz}{\mathfrak{z}}

\newcommand{\so}{\mathfrak{so}}
\newcommand{\su}{\mathfrak{su}}

\newcommand{\gsl}{\mathfrak{sl}}

\newcommand{\oZ}{\omega_{\sst Z}}
\newcommand{\hX}{{\widehat{X}}}
\newcommand{\hY}{{\widehat{Y}}}
\newcommand{\hZ}{{\widehat{Z}}}
\newcommand{\hV}{{\widehat{V}}}
\newcommand{\hA}{{\widehat{A}}}

\newcommand{\st}{\ |\ }

\newcommand{\sst}{\scriptscriptstyle}

\numberwithin{equation}{section}

\title[Closed G$_2$-structures with a transitive reductive group of automorphisms]{Closed G$_{\mathbf2}$-structures with a transitive reductive group of automorphisms}

\author{Fabio Podest\`a}
\address{Dipartimento di Matematica e Informatica ``U.~Dini'' \\ Universit\`a degli Studi di Firenze\\ Viale Morgagni 67/a\\ 50134 Firenze\\ Italy}
\email{fabio.podesta@unifi.it}
\author{Alberto Raffero}
\address{Dipartimento di Matematica ``G. Peano'', Universit\`a  degli Studi di Torino, Via Carlo Alberto 10, 10123 Torino, Italy} 
\email{alberto.raffero@unito.it}

\thanks{The authors were supported by GNSAGA of INdAM and by the project PRIN 2017 ``Real and Complex Manifolds: Topology, Geometry and Holomorphic Dynamics''}

\begin{document}

\maketitle
\begin{center}
{\footnotesize {\em Dedicated to Dmitri V.~Alekseevsky on the occasion of his 80th birthday}}
\end{center}
\begin{abstract} 
We provide the complete classification of seven-dimensional manifolds endowed with a closed non-parallel $\G_2$-structure and admitting a transitive reductive group $\G$ of automorphisms. 
In particular, we show that the center of $\G$ is one-dimensional and the manifold is the Riemannian product of a flat factor and a  non-compact homogeneous six-dimensional manifold endowed with 
an invariant strictly symplectic half-flat SU(3)-structure. 
\end{abstract}

\section{Introduction}
 
A closed $\G_2$-structure on a seven-dimensional manifold $\M$ is given by a definite 3-form $\f$ satisfying the condition $d\f=0$. 
Closed $\G_2$-structures appeared in \cite{FeGr} as one of the sixteen natural classes of $\G_2$-structures, 
and 7-manifolds endowed with these structures provide a fruitful setting where to construct metrics with holonomy $\G_2$ 
(see e.g.~\cite{Bry, CHNP, Joy, JoKa, Kov, KoLe,Lot}). 
However, very little is currently known about general properties of these manifolds, see for instance \cite{Bry, ClIv} for curvature properties, \cite{CoFe,Fer1,Fer2,FFM,FiRa,Lau,LaNi1,LaNi2} 
for examples consisting of Lie groups with left-invariant closed G$_2$-structures, and \cite{FFKM} for a compact example obtained resolving the singularities of a 7-orbifold. 

In our previous work \cite{PR1}, we investigated the properties of the automorphism group $\Aut(\M,\f)\coloneqq \left\{ f\in\Diff(\M) \st f^*\f=\f \right\}$ 
when $\M$ is compact and the closed $\G_2$-structure $\f$ is not parallel with respect to the Levi Civita connection of the corresponding Riemannian metric $g_\f$. 
In particular, we proved that the compact Lie group $\Aut(\M,\f)$ has abelian Lie algebra with dimension bounded above by $\mathrm{min}\{6,b_2(M)\}$. 
As a consequence, we showed that there are no compact homogeneous 7-manifolds endowed with an invariant closed non-parallel $\G_2$-structure, 
i.e., admitting a transitive Lie subgroup $\G\subseteq\Aut(\M,\f)$. 

In the non-compact setting, besides the examples on Lie groups mentioned above, 
it is possible to obtain homogeneous examples on the product of the circle (or the real line) with a non-compact homogeneous 6-manifold endowed 
with an invariant strictly symplectic half-flat SU(3)-structure $(\omega,\psi)$.  
In these examples, the closed $\G_2$-structure is given by $\f = \omega\W ds+ \psi$, where $s$ is the coordinate on the one-dimensional factor, and the transitive Lie group is {\em reductive}, 
i.e., its Lie algebra is the direct sum of a semisimple and an abelian ideal (note that any such Lie group cannot act simply transitively, see \cite{FiRa}). 
This naturally leads to the question whether these examples exhaust the class of such homogeneous manifolds when a reductive group of automorphisms acts transitively.  
Here, we answer this question positively, proving the following result.

\begin{theorem}\label{Main} 
Let $\M$ be a seven-dimensional manifold endowed with a closed non-parallel $\G_2$-structure $\f$, and assume that 
there exists a transitive Lie subgroup $\G\subseteq \Aut(\M,\f)$. If $\G$ is reductive and acts irreducibly on $\M$, then $\M$ is non-compact and 
\begin{enumerate}[\rm i)]
\item the group $\G$ has one-dimensional center and its semisimple part $\G_s$ is (locally) isomorphic to either $\SU(2,1)$ or $\SO(4,1)$;
\item the universal cover of $\M$ is isomorphic to the product $\mathcal O \times \mathbb R$, where $\mathcal O$ is a coadjoint orbit of $\G_s$ 
 	endowed with a $\G_s$-invariant strictly symplectic half-flat structure $(\omega,\psi)$, and the product $\mathcal O \times \mathbb R$ is endowed with the induced $\G_2$-structure $\f=\omega\W ds+\psi$.
\end{enumerate}
\end{theorem}

We recall that a transitive action of a Lie group $\G$ is called {\it irreducible} when no proper normal Lie subgroup of $\G$ acts transitively (see e.g.~\cite[p.~75]{O} for terminology). 
Note that this assumption is not restrictive, as normal subgroups of reductive Lie groups are still reductive (see e.g.~\cite{B}).

We emphasize the following consequence of the above theorem. 
\begin{corollary}
Let $\M$ be a non-compact $7$-manifold endowed with a closed non-parallel $\G_2$-structure $\f$.  
If there exists a transitive Lie subgroup $\G\subseteq \Aut(\M,\f)$, then $\G$ cannot be semisimple. 
\end{corollary}
 
This work is structured as follows. In Section \ref{sectprelim}, we briefly review closed $\G_2$-structures and some related facts about their automorphisms. 
In Section \ref{proofmainthm}, we prove our main Theorem \ref{Main}. 
The proof involves several arguments from the theory of Lie algebras and their representations together with more geometric considerations.

\subsection*{Notation} 
Lie groups and their Lie algebras will be indicated by capital and gothic letters, respectively. 
If a Lie group $\G$ acts on a manifold $\M$, for every $X\in \gg$ we will denote by $\hX$ the vector field on $\M$ induced by the one-parameter subgroup $\exp(tX)$.

The abbreviation $e^{ijk\cdots}$ for the wedge product of covectors $e^i \W e^j \W e^k \W\cdots$ is used throughout the paper.

\section{Preliminaries}\label{sectprelim}

A $\G_2$-structure on a seven-dimensional manifold $\M$ is characterized by the existence of a 3-form $\f\in\Omega^3(\M)$ which is {\em definite}, 
namely at each point $p$ of $\M$ 
\[
0\neq \iota_v\f\W\iota_v\f\W\f\in \Lambda^7(T_p\M)^*,\quad \forall\ v\in T_p\M\smallsetminus\{0\}.
\] 
Such a 3-form $\f$ gives rise to an orientation on $\M$ and to a unique Riemannian metric $g_\f$ 
such that 
\[
g_\f(v,w)\mathrm{vol}_{g_\f} = \frac16\, \iota_v\f\W \iota_w\f\W\f, 
\]
for all $v,w\in T_p\M$.  
A $\G_2$-structure $\f$ is said to be {\em parallel} if $\nabla^{g_\f}\f=0$, where $\nabla^{g_\f}$ denotes the Levi Civita connection of $g_\f$. 
By \cite{FeGr}, this is equivalent to $\f$ being both {\em closed} ($d\f=0$) and {\em coclosed} ($d*_{\f}\f=0$). 
It is well-known that the Riemannian metric induced by a parallel $\G_2$-structure is Ricci-flat. 

We focus on the case when the $\G_2$-structure $\f$ is closed and non-parallel, namely $d\f=0$ and $d*_\f\f\neq0$,  
and we assume the existence of a connected Lie subgroup $\G\subseteq \Aut(\M,\f)$ acting transitively on $\M$. 
Then, $\M$ is necessarily non-compact by \cite[Cor.~2.2]{PR1}, and we can write $\M = \G/\H$, where the isotropy subgroup $\H \coloneqq \G_p$ at some fixed point $p\in\M$ is compactly embedded in $\G$.  
As $\M$ is not compact, we may suppose that it is simply connected and, therefore, that $\H$ is connected.  

From now on, we assume that the group $\G$ is non-compact and reductive. In particular, the Lie algebra $\gg$ of $\G$ can be written as $\gg = \gg_s\oplus\gz$, 
where $\gg_s$ is the Lie algebra of the maximal semisimple connected subgroup $\G_s$ of $\G$ and 
$\gz$ is the center of $\gg$.  
The given $\G_2$-structure $\f$ is not parallel, thus the metric $g_\f$ is not flat, and the Lie algebra $\gg$ is not abelian, i.e., $\gg_s$ is not trivial. 

Since ideals of reductive Lie algebras are also reductive (see e.g.~\cite{B}), we may also suppose that the $\G$-action is irreducible. 
The following lemma, which will be also useful in the sequel, gives further restrictions on the isotropy subgroup.

\begin{lemma}\label{proj} 
Suppose that $\gg = \gs_1\oplus\gs_2$ for some non-trivial ideals $\gs_1,\gs_2$ of $\gg$, and denote by $p_i:\gg\to\gs_i$, $i=1,2$, the corresponding projections. 
Then, 
\[
p_i(\gh)\neq \gs_i,\quad i=1,2.
\]
In particular, the isotropy subgroup $\H$ is contained in the semisimple subgroup $\G_s$.
\end{lemma}
\begin{proof} 
Assume that $p_1(\gh)=\gs_1$, and let $\Sg_2$ be the connected Lie subgroup of $\G$ with Lie algebra $\gs_2$. Then, 
\[
\dim (\Sg_2\cdot p) 	= \dim \gs_2  -\dim (\gh\cap\gs_2) = \dim \gs_2 - (\dim \gh  -\dim \gs_1) = \dim \gg -\dim \gh  = \dim(\M)
\]
implies that $\Sg_2$ has an open orbit in $\M$, hence it is transitive on $\M$ (see e.g.~\cite[p.~178]{KN}), contradicting the irreducibility of the $\G$-action. 
To prove that $\H\subseteq \G_s$, let us consider the projection $pr:\gg\to\gz$ along $\gg_s$,  and let us suppose that $\ga \coloneqq pr(\gh)\neq\{0\}$. 
We then get a contradiction by putting $\gs_1\coloneqq \ga$ and $\gs_2\coloneqq \gg_s\oplus \gb$, 
where $\gb\subseteq \gz$ is a subspace with $\gz = \ga\oplus\gb$. 

\end{proof}

The isotropy subgroup $\H$ is compactly embedded in $\G$ but it is not necessarily compact. We can then consider the closure $\overline\H$ of $\H$ in the isometry group $\mathrm{Iso}(\M,g_\f)$ 
or, equivalently, in the closed subgroup $\Aut(\M,\f)\subseteq\mathrm{Iso}(\M,g_\f)$. Notice that $\overline \H$ is compact, as it is a closed subgroup of the compact group $\mathrm{Iso}(\M,g_\f)_p$,  
and that $\H$ is a normal subgroup of $\overline \H$. Since $\overline \H\subset\Aut(\M,\f)$, it embeds into $\G_2$ via the isotropy representation at $p$. 
We discuss some useful properties of $\H$ and $\overline{\H}$ in the next lemma. 
\begin{lemma}\label{sub}\ 
\begin{enumerate}[i)]
\item\label{Lemi} The Lie algebra $\gh$ is a proper subalgebra of $\gg_2$;
\item\label{Lemii} if the Lie subalgebra $\overline\gh\subset\gg_2$ is proper and non-abelian, then its dimension belongs to $\{3,4,6,8\}$ and it is isomorphic to $\so(3)$, $\gu(2)$, $\so(4)$ and $\su(3)$, respectively. 
In all these cases, $\H$ is compact;
\item\label{Lemiii} if $\H$ is not compact, then $\dim \H =1$.
\end{enumerate}
\end{lemma}
\begin{proof} \ 
\begin{enumerate}[i)] 
\item 
If $\H\cong\G_2$, then $(\M,g_\f)$ is a two-point homogeneous space, as $\G_2$ acts transitively on the unit 6-sphere (see e.g.~\cite{Hel}). 
By \cite[Thm.~3]{Hel1}, $(M,g_\f)$ is then isometric to the Euclidean space $\mathbb R^7$. This implies that $\f$ is parallel. 
\item It is well-known (see e.g.~\cite{OV}) that maximal subalgebras of maximal rank in $\gg_2$ are given by $\su(3)$ or $\so(4)$, while $\so(3)$ appears as the only maximal subalgebra of rank one. 
Moreover, a maximal subalgebra of $\su(3)$ or $\so(4)$ is isomorphic to $\gu(2)$ or $\so(3)$. As $\gh$ is an ideal of $\overline\gh$, our claim follows. 
\item If $\H$ is not compact, then $\overline\gh$ must be either $\gg_2$ or abelian. If $\overline \gh = \gg_2$, it is simple and therefore $\gh$ is trivial by point \ref{Lemi}). 
Consequently, $\overline\gh$ is abelian and $\dim\overline\gh \leq 2$. If $\dim\gh =2$, then $\gh=\overline\gh$ and $\H$ is compact. Thus, the only possibility is $\dim \H = 1$. 
\end{enumerate}
\end{proof}

\section{Proof of the Main Theorem}\label{proofmainthm}
In this section, we prove the Theorem \ref{Main}.  
The proof is divided into three parts, according to $\gg$ being simple, semisimple not simple, and not semisimple.


\subsection{Case $\gg$ simple} \par {}
Using Lemma \ref{sub}, we see that  $\dim(\gh)\in\{0,1,2,3,4,6,8\}$. Therefore, we need to single out those real simple algebras whose dimensions belong to the set 
$\{7,8,9,10,11,13,15\}$.  
We recall that a real simple Lie algebra is either a real form or the realification of a complex simple Lie algebra. 
A direct inspection of the list of complex simple Lie algebras shows that $\dim(\gg) \in \{8,10,15\}$ and that $\gg$ is the real form of one of the following complex simple Lie algebras 
\[
\gsl(3,\mathbb C),~\so(5,\mathbb C),~\gsl(4,\mathbb C).
\]

When $\dim(\gg) = 15$, i.e., $\gg^c = \gsl(4,\mathbb C)$, we see that $\dim(\gh) = 8$, hence $\gh \cong \su(3)$. 
This forces $\gg = \su(3,1)$ with reductive decomposition $\gg = \gh \oplus \gn_o\oplus\gn$, 
where $\gn_o\cong \mathbb R$ and $\gn\cong\mathbb C^3$ is the standard $\SU(3)$-module. 
The $\gh$-invariant $3$-forms on the tangent space $T_p\M\cong\gn_o\oplus\gn$ lie in the modules $\gn_o^*\otimes(\Lambda^2 \gn^*)^\gh$ and $(\Lambda^3\gn^*)^\gh$. 
We may select a basis $\{e_1,\ldots,e_6\}$ of $\gn$ and a basis vector $e_7$ of $\gn_o$ with $e_{2i}=[e_7,e_{2i-1}]$, $i=1,2,3$. 
Let $\{e^1,\ldots,e^6,e^7\}$ be the corresponding dual basis of $(\gn_o\oplus\gn)^*$.
Then, the space $(\Lambda^3\gn^*)^\gh$ is generated by the forms 
\[
\gamma_1\coloneqq e^{135} - e^{146} - e^{236} - e^{245},\quad \gamma_2\coloneqq e^{136} + e^{145} + e^{235} - e^{246}, 
\]
and the module $\gn_o^*\otimes(\Lambda^2 \gn^*)^\gh$ is spanned by $e^{127}+e^{347}+e^{567}$. 
Noting that $[\gn,\gn]\subseteq \gh\oplus \gn_o$, we see that for any closed invariant $3$-form $\phi$ we have 
\[
0= d\phi(e_7,e_1,e_3,e_5)= -\phi(e_2,e_3,e_5) - \phi(e_1,e_4,e_5)-\phi(e_1,e_3,e_6),
\]
so that $\phi$ has no component along $\gamma_2$. 
A similar argument shows that $\phi$ has no component along $\gamma_1$.  
This implies that $\phi$ cannot be definite, and the case $\gg^c = \gsl(4,\mathbb C)$ can be ruled out. \par 

We are then left with the cases $\gg^c = \gsl(3,\mathbb C),\ \so(5,\mathbb C)$. The possible pairs $(\gg,\gh)$ corresponding to these Lie algebras are given in Table \ref{Table1}, 
where it is also specified how $\gh$ sits inside the maximal compactly embedded subalgebra of $\gg$ up to conjugation. 
We discuss each possibility separately. As we will see, in all cases there are no invariant 3-forms that are both closed and definite.  

\par \medskip

\begin{table}[ht]
\centering
\renewcommand\arraystretch{1.4}
\begin{tabular}{|c|c|c|c|}
\hline
n.&$\gh$													& 	$\gg$& note	\\ \hline \hline
1& $\mathbb R$	&	$\gsl(3,\mathbb R)$	&	$\gh\subset\so(3)$	\\ \hline 
2& $\mathbb R$	&	$\su(2,1)$	&	$\gh\subset\gu(2)\subset \su(2,1)$	\\ \hline    
3& $\so(3)$ &	$\so(3,2)$	& $\gh\subset\so(3)\oplus\so(2)\subset \so(3,2)$	\\ \hline    
4& $\so(3)$ &	$\so(4,1)$	& $\gh=\so(3)\oplus\{0\}\subset \so(4)\subset\so(4,1)$		\\ \hline  
5& $\so(3)$ &	$\so(4,1)$ 	& $\gh = \so(3)_{\rm{diag}}\subset \so(3)\oplus\so(3)=\so(4)\subset \so(4,1)$		\\ \hline  

\end{tabular}
\vspace{0.2cm}
\caption{Possible pairs $(\gg,\gh)$ when $\gg^c = \gsl(3,\mathbb C)$ or  $\gg^c=\so(5,\mathbb C)$.}\label{Table1}
\end{table}

\medskip

\subsection*{Case n.1}  We have a reductive decomposition $\gsl(3,\mathbb R) = \gh \oplus \gm$, 
where $\gh \cong \so(2)$. 
Note that $\gh\subset \so(3)$, so that $\H$ is compact and $\gh$ can be assumed to be generated by any element in $\so(3)$ up to conjugation. 
The tangent space $\gm$ splits into the sum of four $\ad(\gh)$-invariant submodules $\gm \cong \bigoplus_{i=0}^3\gm_i$, with $\dim(\gm_0)=1$ and $\dim(\gm_i)=2$, for $i=1,2,3$.
We can fix the following basis of the modules:  
$\gm_0 = {\mbox{Span}}\{e_1\}$, $\gm_i = {\mbox{Span}}\{e_{2i},e_{2i+1}\}$, $i=1,2,3$, and $\gh =  {\mbox{Span}}\{e_8\}$, where 
\begin{eqnarray*}
e_1 = \begin{psmallmatrix}  2&0&0 \\ 0&-1&0\\ 0&0&-1   \end{psmallmatrix},\  e_2 = \begin{psmallmatrix}  0&1&0 \\ -1&0&0\\ 0&0&0  \end{psmallmatrix},\ 
e_3 = \begin{psmallmatrix}  0&0&1 \\ 0&0&0\\ -1&0&0   \end{psmallmatrix},\  e_4 = \begin{psmallmatrix}  0&1&0 \\ 1&0&0\\ 0&0&0\end{psmallmatrix}, & \\ 
e_5 = \begin{psmallmatrix} 0&0&1 \\ 0&0&0\\ 1&0&0  \end{psmallmatrix},\  e_6 = \begin{psmallmatrix}  0&0&0 \\ 0&-1&0\\ 0&0&1   \end{psmallmatrix},\ 
e_7 = \begin{psmallmatrix} 0&0&0 \\ 0&0&1\\ 0&1&0   \end{psmallmatrix},\ 
e_8 = \begin{psmallmatrix} 0&0&0 \\ 0&0&1\\ 0&-1&0 \end{psmallmatrix}. & 
\end{eqnarray*}

The space $(\Lambda^3\gm)^\gh$ has dimension seven and it is generated by the forms 
\begin{eqnarray*}
&\gamma_1 = e^{123},\quad \gamma_2 = e^{145},\quad \gamma_3=e^{167}, & \\
&\gamma_4 = e^{124}+e^{135},\quad \gamma_5 = e^{125} - e^{134}, & \\
&\gamma_6 = e^{246}-e^{257}- e^{347} -e^{356},\quad 
\gamma_7 = e^{247}+e^{256}+e^{346} - e^{357}.& 
\end{eqnarray*}
The generic $\ad(\gh)$-invariant 3-form is then given by $\phi = \sum_{i=1}^7 a_i\,\gamma_i,$ with $a_i\in\R$.  
Using the expression of the Chevalley-Eilenberg differential, we see that 
\[
d\phi(e_3,e_5,e_6,e_7) = -a_3. 
\]
On the other hand, we have
\[
\iota_{e_7}\phi \W \iota_{e_7}\phi \W \phi = 6\left(a_6^2+a_7^2\right) a_3\,e^{1234567}. 
\]
This shows that any closed invariant 3-form $\phi\in(\Lambda^3\gm)^\gh$ cannot be definite.  \par

\subsection*{Case n.2} 
The isotropy subalgebra $\gh\cong\R \subset \gu(2) \subset \su(2,1)=\gg$ can be assumed to be spanned by $e_8 \coloneqq {\mbox{diag}}(i\alpha,i\beta,-i(\alpha+\beta))$, with $\alpha,\beta\in\mathbb{R}$ and $\a^2+\b^2\neq0$. 
Note that $\H$ might be non-compact (cf.~point \ref{Lemiii}) of Lemma \ref{sub}). 
We consider a reductive decomposition $\gg = \gh \oplus \gm$, and we observe that an $\ad(\gh)$-irreducible decomposition of the tangent space $T_p\M\cong\gm$ is given by 
\[
\gm = \gm_0 \oplus \gm_1\oplus \gm_2\oplus\gm_3, 
\]
with $\dim \gm_0=1$ and $\dim \gm_i=2$, $i=1,2,3$. We choose the following basis for the submodules 
\[
\begin{split}
\gm_0:~e_1 = \begin{psmallmatrix} i(-2\,\b-\a)&0&0 \\ 0&i\,(2\,\a+\b)&0\\ 0&0&i\,(\b-\a) \end{psmallmatrix},	&\quad	
\gm_1:~ 	\left\{e_2 = \begin{psmallmatrix} 0&1&0 \\ -1&0&0\\ 0&0&0  \end{psmallmatrix},~ 
		e_3 = \begin{psmallmatrix} 0&i&0 \\ i&0&0\\ 0&0&0  \end{psmallmatrix}\right\},\\
\gm_2:~	\left\{e_4 = \begin{psmallmatrix} 0&0&1 \\ 0&0&0\\ 1&0&0  \end{psmallmatrix},~
		e_5 = \begin{psmallmatrix} 0&0&i \\ 0&0&0\\ -i&0&0  \end{psmallmatrix}\right\},				&\quad
\gm_3:~	\left\{e_6 = \begin{psmallmatrix} 0&0&0 \\ 0&0&1\\ 0&1&0  \end{psmallmatrix},~
		e_7 = \begin{psmallmatrix} 0&0&0 \\ 0&0&i\\ 0&-i&0  \end{psmallmatrix}\right\}, 
\end{split}
\]
and we easily obtain 
\[
\ad(e_8) e_2 = (\a-\b)\,e_3,\quad  \ad(e_8)e_4=(2\a+\b)\,e_5,\quad \ad(e_8)e_6 = (\a+2\b)\,e_7.
\]
The $\ad(\gh)$-module $\gm_0$ is trivial, and the following possibilities occur for $\gm_i$, $i=1,2,3$: 
\begin{enumerate}[a)]
\item\label{aCase} the modules $\gm_i$ are non-trivial and precisely two of them are equivalent;
\item\label{bCase} two of the modules $\gm_i$ are non-trivial and equivalent, while the remaining one is trivial; 
\item\label{cCase} the modules $\gm_i$ are non-trivial and mutually inequivalent. 
\end{enumerate}

Now, in each case \ref{aCase}) - \ref{cCase}) we can determine the expression of the generic closed invariant 3-form on $\gm$ as follows.  
First, we consider the generic invariant 3-form $\phi\in(\Lambda^3\gm^*)^\gh$ and we compute its exterior derivative $d\phi$ using the  expression of the Chevalley-Eilenberg differential  
(and a software for symbolic computations,  e.g.~Maple, if needed). Then, we solve the system of linear equations in the coefficients of $\phi$ arising from the condition $d\phi=0$. 
In detail, we obtain  
\begin{enumerate}[a)]
\item $\dim\left((\Lambda^3\gm)^\gh\right)=7$. If $\gm_1\cong\gm_2$, then the generic closed invariant 3-form is 
\[
\phi	= a_1\left(e^{124}-e^{135} \right) + a_2\left(e^{125}+e^{134}\right) + a_3\left(e^{247}-e^{256}+e^{346}+e^{357} \right), 
\]
and it is not definite. The other possible cases $\gm_1\cong\gm_3$ and $\gm_2\cong\gm_3$ are dealt with similarly. \vspace{0.1cm}
\item $\dim\left((\Lambda^3\gm)^\gh\right)=13$. If $\gm_1$ is trivial, the generic closed invariant 3-form is 
\begin{eqnarray*}
\phi 	&=& a_1 \left(-3(e^{146} + e^{157}) + e^{245} - e^{267} \right) +a_{2} \left(	-3 \left(e^{147}-e^{156} \right) - e^{345}+e^{367} \right)	\\	
	& & +a_3\left(e^{247}-e^{256} + e^{346}+e^{357}\right),
\end{eqnarray*}
and it is not definite. The cases $\gm_2$ or $\gm_3$ trivial are analogous. \vspace{0.1cm}
\item $\dim\left((\Lambda^3\gm)^\gh\right)=5$ and the generic closed invariant 3-form is 
\[
\phi = a_1\left(-e^{247}+e^{256}-e^{346}-e^{357} \right), 
\]
and it is not definite. 
\end{enumerate} 
\par

\subsection*{Case n.3}
We have the Cartan decomposition $\gg = \gk +\gp$, where $\gk = \so(3)+\so(2)$.  
Then, $\gh=\so(3)$ is the semisimple part of $\gk$ and $\gv\coloneqq \so(2)$ is the center of $\gk$. 
The $\ad(\gh)$-module $\gp$ splits as the sum $\gp = \gn_1\oplus\gn_2$, where $\gn_1\cong\gn_2\cong\mathbb R^3$ are equivalent modules. 
We may select a basis $\{e_1,e_2,e_3\}$ of $\gn_1$ and a basis vector $e_7$ of $\gv$ in such a way that $\{e_{i+3}\coloneqq[e_i,e_7]\}_{i=1,2,3}$ is a basis of $\gn_2$. 
The space of $\ad(\gh)$-invariant 3-forms on $T_p\M\cong \gv\oplus\gp$ decomposes into the sum of five one-dimensional submodules as follows
\[
\Lambda^3(\gv\oplus\gp)^\gh = \Lambda^3\gn_1 \oplus (\Lambda^2\gn_1\otimes\gn_2)^\gh \oplus (\gn_1\otimes\Lambda^2\gn_2)^\gh \oplus \left(\gv\otimes(\gn_1\otimes\gn_2)^\gh\right) \oplus \Lambda^3\gn_2. 
\]
From this, we immediately see that a basis of invariant 3-forms is given by
\[
\gamma_1\coloneqq e^{123},~\gamma_2\coloneqq e^{126}-e^{135}+e^{234},~\gamma_3\coloneqq e^{156}-e^{246}+e^{345},~\gamma_4 \coloneqq e^{147}+e^{257}+e^{367},~\gamma_5\coloneqq e^{456}. 
\]
The generic invariant 3-form $\phi=\sum_{i=1}^5a_i\gamma_i$ satisfies 
\[
\iota_{e_7}\phi \W \iota_{e_7}\phi \W \phi = -6\,a_4^3\,e^{1234567},
\]
and we have 
\[
d\phi(e_1,e_2,e_4,e_5) = 2\,a_4.
\]
Consequently, any closed invariant 3-form is not definite. 
\par

\subsection*{Case n.4}
We consider the Lie algebra 
\[
\so(4,1) = \left\{\left(\begin{smallmatrix} 0&v\\{}^tv&A\end{smallmatrix}\right)|\ A\in \so(4),\ v\in \mathbb R^4\right\}, 
\]
and the ideals of $\so(4)\subset\so(4,1)$ given by
\[
\gh 	=	\left \{ \left(\begin{matrix} 0&a&b&c\\-a&0&c&-b\\
                 -b&-c&0&a\\ -c&b&-a&0\end{matrix}\right),\ a,b,c\in\mathbb R\right \},\quad
\gp 	=	\left \{ \left(\begin{matrix} 0&r&s&t\\-r&0&-t&s\\
                 -s&t&0&-r\\ -t&-s&r&0\end{matrix}\right),\ r,s,t\in\mathbb R\right \},
\]
so that $\so(4) = \gh\oplus\gp$, $\gh\cong\so(3)$.  
We have the $\ad(\gh)$-invariant decomposition $\gg = \gh \oplus \gp\oplus \gn$, where $\gn \cong \mathbb R^4$ via the map 
$\gg \ni \left(\begin{smallmatrix} 0&v\\{}^tv&0\end{smallmatrix}\right) \mapsto v \in  \mathbb R^4$. 
Moreover, we have 
\[
[\gh,\gp]= 0,\quad [\gn,\gn]= \gh+\gp,\quad [\gp,\gn]\subseteq \gn.
\]
We select the following basis of $\gp$
\[
e_1 \coloneqq 	\left(\begin{matrix} 0&1&0&0\\-1&0&0&0\\
                 	0&0&0&-1\\ 0&0&1&0\end{matrix}\right),\ 
e_2	\coloneqq \left(\begin{matrix} 0&0&1&0\\ 0&0&0&1\\
                 	-1&0&0&0\\ 0&-1&0&0\end{matrix}\right),\ 
e_3	\coloneqq \left(\begin{matrix} 0&0&0&1\\0&0&-1&0\\
           	      	0&1&0&0\\ -1&0&0&0\end{matrix}\right),
\]
with the standard relations $[e_i,e_j]=2\epsilon_{ijk}e_k$, for $i,j,k\in\{1,2,3\}$. 
Moreover, we consider the canonical basis of $\mathbb R^4\cong\gn$ and we denote it by $\{e_4,e_5,e_6,e_7\}.$

The space of $\ad(\gh)$-invariant $3$-forms $\Lambda^3(\gp\oplus\gn)^\gh$ can be decomposed as 
\[
\Lambda^3(\gp\oplus\gn)^\gh = \Lambda^3\gp \oplus \left(\gp\otimes(\Lambda^2\gn)^\gh\right).
\]
A basis of $(\Lambda^2\gn)^\gh$ is given by 
\[
\o_1 \coloneqq e^{45}-e^{67},\quad \o_2\coloneqq e^{46}+e^{57},\quad \o_3\coloneqq e^{47}-e^{56},
\]
so that a basis of invariant 3-forms is  $\{e^{123},~e^i\wedge\o_j\}_{i,j=1,2,3}.$
We consider the generic invariant 3-form 
\begin{eqnarray*}
\phi	&=&	a_1\,e^{123} + a_2\, e^1\W\o_1 + a_3\,e^1\W\o_2+a_4\,e^1\W\o_3 + a_5\, e^2\W\o_1 + a_6\,e^2\W\o_2+a_7\,e^2\W\o_3\\
	& & + a_8\, e^3\W\o_1 + a_9\,e^3\W\o_2+a_{10}\,e^3\W\o_3,
\end{eqnarray*}
and we notice that 
\[
\iota_{e_1}\phi \W \iota_{e_1}\phi \W \phi = -6\,a_1\left(a_2^2+a_3^2+a_4^2\right) e^{1234567}. 
\]
Now, we have 
\begin{eqnarray*}
d\phi(e_2,e_3,e_6,e_7) = \frac12 a_1 +2a_2-2a_6-2a_{10},	& & d\phi(e_1,e_3,e_5,e_7) = \frac12 a_1 -2a_2+2a_6-2a_{10},\\
d\phi(e_1,e_2,e_5,e_6) = \frac12 a_1 -2a_2-2a_6+2a_{10}, & & d\phi(e_4,e_5,e_6,e_7) = a_2+a_6+a_{10}. 
\end{eqnarray*}
Thus, if $\phi$ is closed, then $a_1=a_2=a_6=a_{10}=0$ and $\phi$ is not definite. 

\subsection*{Case n.5} 
We have the Cartan decomposition $\gg = \gk +\gp$, where $\gk = \so(4)$ and $\gh\cong\so(3)\subset \gk$. 
We can consider the $\ad(\gh)$-invariant decomposition $\gk = \gh \oplus\gn_1$, with $[\gn_1,\gn_1]\subseteq\gh$. 
The $\ad(\gh)$-module $\gp$ splits as $\gp = \gn_2\oplus\gv$, 
where $\gv = \mathbb R V$ for some $V\in \gp$ satisfying the following properties
\[
[V,\gn_1]=\gn_2,\ [V,\gn_2] = \gn_1,\ [\gn_2,\gn_2]\subseteq \gh.
\]
We choose the following basis for $\frh$
\[
e_8 = \begin{psmallmatrix} 0&0&0&0&0 \\ 0&0&0&0&0 \\ 0&0&0&1&0 \\ 0&0&-1&0&0 \\ 0&0&0&0&0  \end{psmallmatrix}, \quad
e_9 = \begin{psmallmatrix} 0&0&0&0&0 \\  0&0&0&0&0 \\ 0&0&0&0&1 \\ 0&0&0&0&0 \\ 0&0&-1&0&0  \end{psmallmatrix}, \quad
e_{10} =  \begin{psmallmatrix} 0&0&0&0&0 \\  0&0&0&0&0 \\ 0&0&0&0&0 \\ 0&0&0&0&1 \\ 0&0&0&-1&0  \end{psmallmatrix},
\]
and the following basis for the irreducible summands of the tangent space $\gm \cong \gn_1\oplus\gn_2\oplus\gv$
\begin{eqnarray*}
\gn_1:&~& 
e_1 = \begin{psmallmatrix} 0&0&0&0&0 \\ 0&0&0&0&1 \\ 0&0&0&0&0 \\ 0&0&0&0&0 \\ 0&-1&0&0&0  \end{psmallmatrix}, \quad
e_2 = \begin{psmallmatrix} 0&0&0&0&0 \\ 0&0&0&1&0 \\ 0&0&0&0&0 \\ 0&-1&0&0&0 \\ 0&0&0&0&0  \end{psmallmatrix}, \quad
e_3 = \begin{psmallmatrix} 0&0&0&0&0 \\ 0&0&1&0&0 \\ 0&-1&0&0&0 \\ 0&0&0&0&0 \\ 0&0&0&0&0  \end{psmallmatrix}, \\
\gn_2:&~&
e_4 = \begin{psmallmatrix} 0&0&0&0&1 \\ 0&0&0&0&0 \\ 0&0&0&0&0 \\ 0&0&0&0&0 \\ 1&0&0&0&0  \end{psmallmatrix}, \quad
e_5 = \begin{psmallmatrix} 0&0&0&1&0 \\ 0&0&0&0&0 \\ 0&0&0&0&0 \\ 1&0&0&0&0 \\ 0&0&0&0&0  \end{psmallmatrix}, \quad
e_6 = \begin{psmallmatrix} 0&0&1&0&0 \\ 0&0&0&0&0 \\ 1&0&0&0&0 \\ 0&0&0&0&0 \\ 0&0&0&0&0  \end{psmallmatrix}, \\
\gv:&~& 
e_7 = \begin{psmallmatrix} 0&1&0&0&0 \\ 1&0&0&0&0 \\ 0&0&0&0&0 \\ 0&0&0&0&0 \\ 0&0&0&0&0  \end{psmallmatrix}. 
\end{eqnarray*} 
We then see that $(\Lambda^3\gm)^\gh$ splits into the sum of the following one-dimensional submodules
\[
(\Lambda^3\gm)^\gh \cong 	\Lambda^3\gn_1 \oplus \Lambda^3\gn_2 \oplus (\Lambda^2\gn_1 \otimes \gn_2)^\gh \oplus (\gn_1 \otimes \Lambda^2\gn_2)^\gh 
						\oplus (\gn_1\otimes\gn_2\otimes\gv)^\gh, 
\]
and a basis of invariant 3-forms is given by 
\[
\gamma_1 = e^{123},~\gamma_2 = e^{456},~\gamma_3 = e^{126}-e^{135}+e^{234},~\gamma_4 = e^{156}-e^{246}+e^{345},~\gamma_5 = e^{147}+e^{257}+e^{367}. 
\]
The generic invariant 3-form $\phi = \sum_{i=1}^5 a_i\gamma_i$ satisfies 
\[
\iota_{e_7}\phi \W \iota_{e_7}\phi \W \phi = -6\,(a_5)^3\,e^{1234567}. 
\]
If $\phi$ is closed, then 
\[
0 = d\phi(e_1,e_2,e_4,e_5) = d\phi(e_1,e_3,e_4,e_6) = d\phi(e_2,e_3,e_5,e_6) = 2\,a_5. 
\]
Thus, no closed invariant 3-form on $\gm$ can be definite.


\subsection{Case $\gg$ semisimple not simple}
We begin with the following. 
\begin{lemma} Let $\gg$ be semisimple, not simple. Then, any simple factor of $\gg$ appears in the following list: 
\begin{enumerate}[-]
\item real forms of $\gsl(2,\mathbb C),~\gsl(3,\mathbb C),~\so(5,\mathbb C)$;
\item the realification of $\gsl(2,\mathbb C)$.
\end{enumerate}
Moreover, the dimension of the isotropy subalgebra $\gh$ belongs to $\{2,4,6\}$.
\end{lemma}
\begin{proof} 
By Lemma \ref{sub} we have that $\dim(\gh)\leq 8$, whence $\dim(\gg)\leq 15$. 
The possible simple factors of $\gg$ can be deduced from a direct inspection of the list of complex simple Lie algebras with dimension at most 15 (see e.g.~\cite[Ch.~X, Table IV]{Hel}), 
together with the fact that $\gg$ is not simple. 
From this, we see that $\dim(\gh) \geq 2$, as there are no semisimple real algebras of dimension $7$ and a semisimple real Lie algebra of dimension $8$ is simple. 
Noting that $\dim(\gh) = 3$ cannot occur as $\dim(\gg) = 10$ implies that $\gg$ is simple, we have that  either $\dim(\gh) = 2$ or $\dim(\gh) \geq 4$. 
The dimensions $\dim(\gh) = 5,7$ are ruled out using Lemma \ref{sub}, as $\gh$ is an ideal of $\overline\gh$. 
If $\dim(\gh) = 8$, again by Lemma \ref{sub} we have that $\gh \cong\su(3)$ and $\dim(\gg) = 15$. 
We write $\gg = \bigoplus_{i=1}^k\gs_i$ as a sum of its simple factors, with $\dim(\gs_i)\in \{3,6,8,10\}$, $i=1,\ldots,k$, $k\geq 2$. As $\gh$ is simple, 
at least one simple factor of $\gg$ has dimension $8$ or $10$ so that $\sum_{i=1}^k\dim\gs_i\neq 15$. Therefore, $\dim(\gh)=8$ is excluded. 

\end{proof}

The previous lemma allows us to describe all possibilities for the pair $(\gg,\gh)$. They are listed in Table \ref{Table2}. \par

\begin{table}[ht]
\centering
\renewcommand\arraystretch{1.4}
\begin{tabular}{|c|c|c|c|}
\hline
n.&	$\gh$		& 	$\gg$			& note														\\ \hline \hline
1&	$2\mathbb R$	&	$3\gs$			&	$\gs^c = \gsl(2,\mathbb C)$									\\ \hline 
2& 	$2\mathbb R$	&	$\gs_1\oplus\gs_2$	&	$\gs_1^c = \gsl(2,\mathbb C), \gs_2=\gsl(2,\mathbb C)_{\mathbb R}$	\\ \hline    
3& 	$\gu(2)$ 		&	$\gs_1\oplus\gs_2$	&	 $\gs_1^c = \gsl(2,\mathbb C), \gs_2^c=\gsl(3,\mathbb C)$			\\ \hline    
4& 	$\so(4)$ 		&	$\gs_1\oplus\gs_2$	& 	$\gs_1^c = \gsl(2,\mathbb C), \gs_2^c=\so(5,\mathbb C)$				\\ \hline  
\end{tabular}
\vspace{0.2cm}
\caption{Possible pairs $(\gg,\gh)$ when $\gg$ is semisimple and not simple.}\label{Table2}
\end{table}

The following proposition rules out the cases n.2 and n.4 of Table \ref{Table2}. We will deal with the remaining pairs separately. 

\begin{proposition} 
The pairs $(\gg,\gh)$ appearing as n.$2$ and n.$4$ in Table \ref{Table2} correspond to homogeneous spaces with no invariant $\G_2$-structures. 
\end{proposition}

\begin{proof} 
In case n.2, the isotropy subgroup $\H$ embeds as a maximal torus in $\G_2$. Hence, the tangent space $T_p\M$ contains three inequivalent real $2$-dimensional $\H$-modules. 
Now, as any abelian subspace of $\gs_1$ is one-dimensional, we see that $\gh$ projects onto a non trivial compactly-embedded subalgebra $\gl$ of $\gs_2$. 
Up to an inner automorphism, we may suppose that $\gl\subset \gu \cong \su(2)$, where $\gs_2 = \gu \oplus i\gu$ is a Cartan decomposition. 
Hence, $\gh \subseteq \gs_1\oplus \gl = \gh\oplus \gq,$ for some $\ad(\gh)$-invariant submodule $\gq$. 
Considering an $\ad(\gh)$-invariant decomposition $\gu = \gl \oplus \gn$, we see that $\gg = \gh \oplus \gq \oplus \gn \oplus i\gu$, 
showing that the isotropy representation of $\gh$ contains $\gn$ with multiplicity two, a contradiction.  

In case n.4, the projection of $\gh$ into $\gs_1$ is not surjective by Lemma \ref{proj} and, therefore, it is trivial. 
Thus, the linear isotropy representation has a fixed point set of dimension at least $3$. On the other hand, the existence of an invariant $\G_2$-structure implies that $\gh$ embeds into $\gg_2$,  
and the fact that $\SO(4)\subset \G_2$ has trivial fixed point set in $\mathbb R^7$ gives a contradiction. 

\end{proof}

In the following propositions, we consider the remaining cases n.1 and n.3.

\begin{proposition} 
In case n.$1$, there exists no invariant closed $\G_2$-structure.
\end{proposition}
\begin{proof}  
Let $\frg = \gs_1\oplus\gs_2\oplus\gs_3$, where $\gs_j^c \cong \gsl(2,\mathbb C)$, and suppose there exists an invariant $\G_2$-structure. 
It then follows that the isotropy $\gh$ can be realized as a maximal abelian subalgebra of $\gg_2$. 
Hence, as an $\gh$-module we have $T_p\M \cong V_o \oplus \bigoplus_{j =1}^3 V_j$, where $V_o$ is a one-dimensional trivial module, 
while $V_j \cong \mathbb R^2$, $j = 1,2,3,$ are mutually inequivalent irreducible submodules. 
This implies that each projection of $\gh$ into the simple factors $\gs_j$ of $\gg$ is not trivial, otherwise the isotropy representation would have a trivial module of dimension at least three. 

If we select $A\coloneqq\mbox{diag}(i,-i)\in\gsl(2,\mathbb C)$, we can suppose that 
\[
\gh = \left\{(\a_1(v)A,\a_2(v)A,\a_3(v)A)\in \gs_1\oplus\gs_2\oplus \gs_3 \st v\in \gh \right\},
\] 
for suitable nonzero $\a_j\in \gh^*$, $j=1,2,3$, with $\a_j\neq \pm\a_k$ if $j \neq k$, and $\sum_j\a_j = 0$, as $\gh$ embeds into $\su(3)\subset\gg_2$. 
We also fix $\gp_j\cong \mathbb R^2$ in $\gs_j$, with $\gs_j = \mathbb R A \oplus \gp_j$ being a Cartan decomposition. 

Now, if we set $V \coloneqq (A,A,A)\in \gg$ and $\gv \coloneqq \mathbb R V$, then 
$\gg = \gh \oplus \gv \oplus \bigoplus_{j=1}^3 \gp_j$ and the tangent space $T_p\M$ identifies with $\gm\coloneqq \gv \oplus \bigoplus_{j=1}^3 \gp_j$. 
Consequently, we have 
\[
(\Lambda^3\gm)^\gh \cong \left(\bigoplus_{j=1}^3\gv\otimes \Lambda^2\gp_j\right) \oplus (\gp_1\otimes\gp_2\otimes\gp_3)^\gh.
\]
We now fix a basis $\{e_1,\ldots,e_7\}$ of $\gm$ with $e_7 \coloneqq V$, $e_1,e_2\in \gp_1$, $e_3,e_4\in \gp_2$ and $e_5,e_6\in \gp_3$ so that 
$\ad(V)|_{\gp_j}= \left(\begin{matrix} 0&-2\\ 2&0\end{matrix}\right)$, for $j=1,2,3$. 
Then, with respect to the dual basis $\{e^1,\ldots,e^7\}$, the forms 
\begin{eqnarray*}
&\gamma_1 \coloneqq e^{127},\quad \gamma_2 \coloneqq e^{347},\quad \gamma_3 \coloneqq e^{567},& \\
&\gamma_4\coloneqq e^{135}-e^{146}-e^{236}-e^{245},\quad \gamma_5 \coloneqq e^{145} + e^{136} + e^{235}-e^{246},&
\end{eqnarray*}
span the space of $\ad(\gh)$-invariant $3$-forms on $\gm$. Any such $\phi$ can be written as $\phi = \sum_{j=1}^5a_j \gamma_j$, for $a_j\in \mathbb R$.
If $\phi$ is closed, then
\[
0 = d\phi(e_7,e_1,e_3,e_5) = -2(\phi(e_2,e_3,e_5)+\phi(e_1,e_4,e_5) + \phi(e_1,e_3,e_6)) = -6a_5. 
\]
Similarly, we get $a_4 = 0$. Therefore, $\phi \in \Span(\gamma_1,\gamma_2,\gamma_3)$ and it cannot be definite.

\end{proof}

\begin{proposition} 
In case n.$3$, there exists no invariant closed $\G_2$-structure.
\end{proposition}
\begin{proof} 
Suppose there exists an invariant closed $\G_2$-structure. 
We let $\gh \coloneqq \gh_s \oplus \mathbb{R} Z$, where $\gh_s \cong \su(2)$ and $Z$ generates the center of $\gh$, and we denote by $\mbox{pr}_j:\gg\to\gs_j$, $j=1,2$, 
the projections onto the simple factors of $\frg$.   
By Lemma \ref{proj}, we may suppose that $\mbox{pr}_1(\gh)\neq \gs_1$. We claim that $\mbox{pr}_1(\gh)\neq \{0\}$. 
Indeed, if $\gh\subset \gs_2$, then the fixed point set of the isotropy representation would be at least $3$-dimensional, while $\H$ embeds as a maximal rank subgroup of  $\G_2$, 
whence its fixed point set is at most one-dimensional. 
Therefore, $\mbox{pr}_1(\gh_s)=\{0\}$ and $A \coloneqq \mbox{pr}_1(Z)  \neq 0$. 
We can also suppose that $A = {\mbox{diag}}(i,-i)\in \gsl(2,\mathbb C)$, with $\gs_1 = \mathbb R A \oplus \gp$ a Cartan decomposition. 

We now claim that $\mbox{pr}_2(Z)\neq 0$. 
Indeed, otherwise $Z=A\in \gs_1$ would act trivially on a five-dimensional subspace of the tangent space $T_p\M\cong\gg/\gh$.  
The torus $\mathrm{T}^1$ generated by $Z$ embeds into $\G_2$, hence into $\SU(3)\subset\G_2$ up to conjugation. 
As any non-trivial element of $\SU(3)$ has a fixed point set in $\mathbb{C}^3$ of complex dimension at most one, we see that the fixed point set of $\mathrm{T}^1$ in $\R^7$ has real dimension at most three. 
This gives a contradiction. 

The ideal $\gs_2$ is isomorphic to one of $\su(3), \su(1,2), \gsl(3,\mathbb{R})$, and we claim that the last possibility cannot occur. 
Indeed $\gh_s\cong \so(3) \subset \gs_2 \cong \gsl(3,\mathbb{R})$ would have a trivial centralizer in $\gs_2$, while $\mathrm{pr}_2(Z)\neq 0.$ 
Therefore, we can suppose that $\gh_s = \left\{\mbox{diag}(0,A)\in \gsl(3,\mathbb C) \st A\in \su(2)\right\}$ and $B \coloneqq \mbox{pr}_2(Z) = \mbox{diag}(2ia,-ia-ia)\in \gsl(3,\mathbb C)$, 
for some nonzero $a\in \mathbb R$. 
Then, we can fix an $\ad(\gh)$-invariant decomposition $\gs_2 = (\gh_s\oplus\mathbb R B) \oplus \gn $, 
and we may consider some nonzero $V\in \mbox{Span}\{A,B\}$ so that 
\[
\gg = \gh \oplus \gv  \oplus \gp \oplus \gn,\quad \gv \coloneqq \mathbb{R} V,
\] 
is an $\ad(\gh)$-invariant decomposition of $\gg$ and $T_p\M$ can be identified with $\gm \coloneqq \gv \oplus \gp \oplus \gn$.  
We choose 
\[
V = \mbox{diag}(i,-i)\oplus\mbox{diag}(2bi,-bi,-bi), \quad b\neq 0,a.  
\]
We let $e_7\coloneqq V$, and we select a basis $\{e_1,\ldots,e_4\}$ of $\gn$ and a basis $\{e_5,e_6\}$ of $\gp$ so that 
\[
\begin{split}
& [e_1,e_2]_\gm = \frac{\eta}{b-a}\, e_7,\quad  [e_3,e_4]_\gm =  \frac{\eta}{b-a}\, e_7,\\  
& [e_5,e_6]_\gm = \frac{2a \varepsilon}{a-b} e_7,\quad  [e_i,e_j]_\gm = 0,~i=1,2, j=3,4,
\end{split}
\]
\[
\ad(e_7)|_\gp = \ad(Z)|_\gp =  \left(\begin{matrix} 0&-2\\2&0\end{matrix}\right),
\quad 
\ad(e_7)|_{\gn}  = \left(\begin{matrix} 0&-3b&0&0\\3b&0&0&0\\ 0&0&0&-3b\\ 0&0&3b&0\end{matrix}\right) = \frac{b}{a}\ad(Z)|_\gn, 
\]
where $\varepsilon,\eta = \pm 1$ according to the Lie algebras $\gs_1,\gs_2$ being of compact or non-compact type. 
We have the following $\ad(\gh)$-invariant decomposition
\[
(\Lambda^3\gm)^\gh = (\gv\otimes\Lambda^2\gp)\oplus (\gv\otimes(\Lambda^2\gn)^\gh) \oplus (\gp\otimes\Lambda^2\gn)^\gh. 
\]
A straightforward computation shows that 
\[
\dim (\gp\otimes\Lambda^2\gn)^\gh  = 
\begin{cases}
2, \quad \mbox{if } a =\frac13,\\
0, \quad \mbox{otherwise}. 
\end{cases}
\]
Let us denote by $\{e^1,\ldots,e^7\}$ the dual $1$-forms. When $a\neq\frac13$, a basis of $(\Lambda^3\gm)^\gh$ is given by  
\[
\gamma_1\coloneqq e^{567},\quad \gamma_2 \coloneqq e^{127}+e^{347}.  
\]
Clearly, these forms do not span any definite 3-form. When $a=\frac13$, a basis of invariant 3-forms is given by
\[
\gamma_1,\quad \gamma_2,\quad \gamma_3 \coloneqq e^{135}+e^{146}+e^{236}-e^{245},\quad \gamma_4 \coloneqq e^{136}-e^{145}-e^{246}-e^{235}. 
\]
In this case, the generic $\ad(\gh)$-invariant $3$-form $\phi$ can be written as $\phi=\sum_{j=1}^4c_j \gamma_j$, for some $c_j\in \mathbb R$. 
If $\phi$ is closed, then
\[
0 = d\phi(e_7,e_5,e_1,e_3) = -2\,\phi(e_1,e_3,e_6) -3b\,\phi(e_2,e_3,e_5) -3b\,\phi(e_1,e_4,e_5)   = (6b-2)\, c_4, 
\]
whence $c_4=0$, as $b\neq a=\frac13$. Similarly, from $d\phi(e_1,e_3,e_6,e_7) = 0$, we obtain $c_3=0$. It then follows that $\phi$ is not definite. 

\end{proof}


\subsection{Case $\gg$ not semisimple}  
In this last case, we have $\gg = \gg_s \oplus \gz$, with $\gz$ non-trivial. 
We start noting that for every $Z\in \gz$ the $2$-form $\oZ \coloneqq\iota_{{\hZ}}\f$ is $\G$-invariant and closed, 
as $d\oZ = \mathcal{L}_\hZ \f - \iota_\hZ d\f = 0$. 
Consequently, if $X,Y\in\gg_s$ and $V\in \gz$, we have 
\[
0 = d\oZ \left(\hX,\hY,\hV\right) =  \oZ \left( \left[\hX,\hY\right],\hV \right). 
\]
As $\gg_s$ is semisimple, it satisfies $\gg_s = [\gg_s,\gg_s]$. Therefore 
\beq\label{fund}
\oZ \left(\hA,\hV\right) = 0,\quad \forall~A\in \gg_s,~V\in \gz.
\eeq

Let $\G_s$ denote the connected Lie subgroup of $\G$ with Lie algebra $\gg_s$. 
The $\G_s$-orbit $\mathcal{O} \coloneqq \G_s\cdot p \cong \G_s/\H$ is a proper submanifold of $\M$, and we may select a nonzero $Z\in\gz$ so that $\hZ_p \not\in T_p\mathcal{O}$.  
We claim that the pull-back of $\oZ$ to $\mathcal{O}$ is an invariant symplectic form. 
Indeed, if $X\in \gg_s$ satisfies $\oZ\left(\hX_p,\hY_p\right) = 0$ for all $Y\in \gg_s$, then from \eqref{fund} we see that $\hX_p$ must lie in the kernel of $\oZ|_p$. 
Thus, $\hX_p$ must be a multiple of $\hZ_p$, hence zero. Therefore, $\dim\mathcal{O}\in\{2,4,6\}$.  

\begin{lemma} 
The orbit $\mathcal{O}$ has dimension six.
\end{lemma}
\begin{proof} 
We first prove that $\dim \mathcal O\geq 4$. Let $\hat\gz|_p\coloneqq \left\{\hV_p\in T_p\M \st V\in\gz\right\}$.  
If $\dim \mathcal O=2$, then we have $\dim  \hat\gz|_p =  5$. 
Given a nonzero $X\in \gg_s$, by \eqref{fund} we know that $\iota_{\hX}\f(V_1,V_2) = 0$ for all $V_1,V_2\in \hat\gz|_p$. 
Now $\iota_{\hX}\f$ is non-degenerate on $U \coloneqq {\mathrm{Span}}\{\hX_p\}^\perp$ and it vanishes on the subspace $U \cap \hat\gz_p$, which has dimension at least $4$, a contradiction.

Suppose now that $\dim\mathcal O=4$.  
Since $\gh\subset\gg_s$ and $\mathcal O$ is symplectic, $\gh$ is not trivial (see e.g.~\cite{BFR}), and a maximal abelian subalgebra of $\gh$ has dimension ${\rm{rk}}(\gg_s^c)\geq 2$. 
It follows that there exists a $2$-torus $\rm T^2$ in $\H$ whose fixed point set in $T_p\M$ has dimension $3$. On the other hand, $\rm T^2$ embeds as a maximal torus of $\G_2$, 
which has a one-dimensional fixed point set in $\mathbb{R}^7,$ a contradiction.

\end{proof}

Since the $\G$-action is irreducible and $\dim\mathcal{O}=6$, we necessarily have $\dim(\gz) = 1$, so that $\gz=\mathrm{Span}\left\{Z\right\}$. 
We claim that $\hZ_p$ belongs to the orthogonal complement of $T_p\mathcal{O}$ in $T_p\M$ with respect to the invariant Riemannian metric $g_\f$. 
Indeed, we may consider an $\ad(\gh)$-invariant decomposition $\gg_s = \gh \oplus \gm$, where $\gm\cong T_p\mathcal O$. 
The invariant symplectic form on $\mathcal O$ corresponds to an $\Ad(\H)$-invariant symplectic form on $\gm$ which can be written as $B(\ad(W)\cdot,\cdot)$ for a unique $W\in \gg_s$, 
where $B$ denotes the Cartan-Killing form of $\gg_s$. Moreover, $\gh$ coincides with the centralizer of $W$ in $\gg_s$ (see, for instance, \cite{BFR}).  
Consequently, we have $\gm^{\H} = \{0\}$.  
This implies that the orthogonal projection of $\hZ_p$ on $T_p\mathcal{O}$, being invariant under $\ad(\gh)$, is trivial. 
 
Let $\psi$ denote the closed $\G_s$-invariant 3-form on $\mathcal{O}$ obtained by pulling back the invariant closed $\G_2$-structure $\f$ on $\M$. 
To conclude the proof of the main theorem, we need to show that the pair $(\oZ,\psi)$ defines a $\G_s$-invariant SU(3)-structure 
on the six-dimensional homogeneous space $\mathcal{O}=\G_s/\H$. Since both $\oZ$ and $\psi$ are closed and $\f$ is not parallel, the SU(3)-structure will be strictly symplectic half-flat, namely  
$d*\psi\neq0$, where $*$ is the Hodge operator relative to the metric induced by $(\oZ,\psi)$. 
In particular, the orbit $\mathcal O$ is non-compact (see \cite[Prop.~4.2]{PR2}).

Now, identifying the invariant closed $\G_2$-structure $\f$ on $\M$ with the corresponding $\ad(\gh)$-invariant definite 3-form $\f$ on $\gm \obot \gz \cong T_p\M,$ we see that 
$\f = \oZ \W \eta +\psi,$ where $\eta\in\gz^*$ is dual to $Z$. Since $\f$ is definite, the pair of $\ad(\gh)$-invariant forms $(\oZ,\psi)$ on $\gm$ defines an SU(3)-structure. 

Summing up, the orbit $\mathcal{O}=\G_s\cdot p$ is a non-compact $\G_s$-homogeneous six-dimensional manifold endowed with an invariant strictly symplectic half-flat SU(3)-structure. 
By the classification result \cite[Thm.~5.1]{PR2}, we have that the pair $(\G_s,\H)$ is (locally) isomorphic to either $(\SO(4,1),\U(2))$ or $(\SU(2,1),\mathrm{T}^2)$.  
We recall that the classification of all invariant strictly symplectic half-flat SU(3)-structures on these homogeneous spaces is also given in \cite[Thm.~5.1]{PR2}.

\end{document}